\documentclass[11pt]{article}
\usepackage{multicol}
\usepackage{amsfonts}
\usepackage{amsmath,amssymb,amsthm,amsfonts,mathrsfs}
\usepackage{graphicx}
\usepackage{epstopdf}
\usepackage{epsfig,subfig,color}
\usepackage{float}
\DeclareGraphicsExtensions{.pdf,.eps,.png,.jpg,.mps}
\usepackage{tkz-graph}
\usetikzlibrary{calc}
\usetikzlibrary{decorations.markings}

\newtheorem{theorem}{Theorem}

\newtheorem{proposition}[theorem]{Proposition}

\newtheorem{lemma}[theorem]{Lemma}

\numberwithin{equation}{section}

\newcommand{\ba}{\begin{array}}
\newcommand{\ea}{\end{array}}

\newcommand{\R}{{\mathbb R}}

\begin{document}
\openup 1.13\jot

\title{Minimum number of non-zero-entries in a $7\times 7$ stable matrix\thanks{Partially supported by NSF grant DMS-1331021. }}

\author{Christopher Hambric\textsuperscript{1}, Chi-Kwong Li\textsuperscript{1},
Diane Christine Pelejo\textsuperscript{2}, Junping Shi\textsuperscript{1} \\\\
{\small \textsuperscript{1} Department of Mathematics, College of William and Mary,\hfill{\ }}\\
\ \ {\small Williamsburg, Virginia, 23187-8795, USA\hfill {\ }}\\
{\small \textsuperscript{2} Institute of Mathematics, College of Science, University of the Philippines Diliman\hfill{\ }}\\
\ \ {\small Diliman, Quezon City 1101, Philippines\hfill {\ }}}
\date{}
\maketitle

\begin{abstract}
We prove that if a $7\times 7$ matrix is potentially stable, then it has at least $11$ non-zero entries. The results for $n\times n$ matrix with $n$ up to $6$ are known previously. We prove the result by making a list of possible associated digraphs with at most $10$ edges, and then use algebraic conditions to show all of these digraphs or matrices cannot be potentially stable. 
In relation to this, we also determine the minimum number of edges in a strongly connected digraph depending on its circumference.
\end{abstract}
\section{Introduction}

%\subsection{Motivation}
%Consider the following linear system of $n$ differential equations with $n$ variables:
%$$
%\left\lbrace\begin{array}{c}
%x_1'=a_{11}x_1+\ldots+a_{1n}x_n\\
%\vdots\\
%x_n'=a_{n1}x_1+\ldots+a_{nn}x_n\\
%\end{array}\right.
%$$
The concept of stability of equilibrium is central to the studies of differential 
equations. By using the techniques of linearization and transforming the equilibrium 
to zero, the stability problem is reduced to $u'=Au$, where $u\in \R^n$ and $A$ is a 
real-valued  $n\times n$ matrix. The equilibrium $u=0$ is asymptotically stable if 
each solution $u$ of $u'=Au$ converges to zero as $t\to\infty$. From the theory of 
linear differential equation, this is equivalent to that each eigenvalue 
of $A$ has negative real part. Hence it is desirable to know what kind of matrices 
are stable, and how to design a matrix to be stable \cite{Maybee1969}.

%\subsection{The Problem}
Let $M_n$ be the set of all $n\times n$ matrices with real-valued entries. A matrix $A\in M_n$ is said to be \textit{stable} if, for each of its eigenvalues $\lambda_1,\lambda_2,\ldots,\lambda_n$, ${\rm Re}(\lambda_i)<0$. A system which is modeled by such a matrix $A$ has stable equilibria, and given small perturbations of its initial conditions the system will return to these equilibrium points.

We define the \textit{sign pattern} of a matrix $A=[a_{ij}]$ to be an $n\times n$ matrix  $S(A)=[s_{ij}]$ such that, for $i,j\in\{1,\ldots,n\}$, $s_{ij}=0$ when $a_{ij}=0$, $s_{ij}=-$ when $a_{ij}<0$, and $s_{ij}=+$ when $a_{ij}>0$. If some matrix $A\in M_n$ is found to be stable, then the sign pattern $S(A)$ is said to be \textit{potentially stable}, or \textit{PS} for short.
In the case where $A\in M_n$ is an upper triangular, lower triangular, or diagonal matrix, or when $A$ is permutationally similar to such a matrix, the problem becomes trivial due to the ease of calculating the eigenvalues of these matrices. Therefore we restrict our examination to \textit{irreducible} matrices, or the matrices $A\in M_n$ such that there does not exist a permutation matrix $P$ such that
$$PAP^T=
\begin{bmatrix}
A_{11} & 0\\
A_{12} & A_{22}
\end{bmatrix}, A_{11}\in M_k, A_{22}\in M_{n-k}.$$

The following result has been proved in \cite{Grundy}.

\begin{theorem}\label{thm1} Let the minimum number of nonzero entries required for an $n\times n$ irreducible sign pattern to be potentially stable be given by $m_n$. Then
$$
\left\lbrace\begin{array}{l l}
m_n=2n-1, & n=2,3,\\
m_n=2n-2, & n=4,5,\\
m_n=2n-3, & n=6,\\
m_n\leq 2n-(\lfloor\frac{n}{3}\rfloor+1), & n\geq 7.
\end{array}\right.
$$
\end{theorem}
Hence the value of $m_n$ for $n=2,3,4,5,6$ was determined in Theorem \ref{thm1}, as well as an upper bound for $m_n$ for any $n\geq 6$ via an explicit construction. Previously other partial results have been obtained for the cases $3\le n\le 5$ \cite{Johnson1989,Johnson1997}. 
In this paper we prove the following theorem:
\begin{theorem}\label{thm2}
$$
m_7=2(7)-3=11.
$$
\end{theorem}
Note that Theorem \ref{thm1} has shown that $11$ is an upper bound. So here in order to prove this minimum, we need only show that there cannot exist a potentially stable $7\times7$ sign pattern with only $10$ nonzero entries. Note that if there were a potentially stable $7\times7$ sign pattern with fewer than $10$ nonzero entries, then we would similarly be able to construct a potentially stable pattern with $10$ nonzero entries by adding additional nonzero entries to an existing potentially stable pattern. Thus it is sufficient to prove that no potentially stable pattern with only $10$ nonzero entries exists.

In order to prove that no such sign pattern exists, we first construct a list of all digraphs with $7$ vertices and $10$ edges which allow for correct minors (as defined by the relationship between cycles in the graph and the minors of the associated matrix in subsection 2.2). This construction is given in Section 3. Once we have constructed this list of digraphs, we will construct the associated set of nonequivalent matrix sign patterns which have correct minors. Fore that purpose we utilize a variant of Routh-Hurwitz stability criterion to show that none of these candidate sign patterns have a stable realization (see Section 4). From this we will conclude that the minimum number of nonzero entries must be equal to $11$.

%\textcolor{blue}{Introduction still needs more work}

\section{Preliminaries}

\subsection{Digraphs}
%%%%%%%%%%%%%%
%Similarly, the \textit{signed digraph} of $A$, denoted by $SD(A)$ is simply the digraph of $A$ with its edges labeled according to the sign of the corresponding entry of $A$. For example,
%$$
%A=\begin{bmatrix}
%1 & -1 & 0 \\
%1 & 0 & 1 \\
%0 & -1 & 0
%\end{bmatrix} \quad \Longrightarrow \quad
%digraph(A)=
%\begin{minipage}{0.2\textwidth}\begin{tikzpicture}[->,>=stealth',shorten >=1pt,auto,node distance=2cm,thick,main node/.style={circle, fill opacity=0.75,minimum size = 6pt, inner sep = 0pt}]

%  \node[main node] (1) {1};
%  \node[main node] (2) [below left of=1] {2};
%  \node[main node] (3) [below right of=1] {3};

 % \path
 %   (1) edge [loop above] node {$+$} (1)
 %       edge node[right] {$-$} (2)
 %   (2) edge[bend left] node[left] {$+$} (1)
 %       edge node {$+$} (3)
 %   (3) edge [bend left] node[above] {$-$} (2);

%\end{tikzpicture}\end{minipage}
%$$

%\textcolor{blue}{signed digraph is not used in paper, so I remove the definition. Some definitions below could be worded better. Also definition of cycle and simple cycle should appear here.}

We define the \textit{digraph} of an $n\times n$ matrix $A=(a_{ij})$ to be a directed graph with vertex set $V_n=\{1,\ldots,n\}$, and for each $i,j\in V_n$, there exists an edge from vertex $i$ to vertex $j$ if and only if $a_{ij}\neq 0$.
For a digraph, we define a \textit{path} as an ordered set of edges such that, for some vertices $i,j,l\in\{1,\ldots,n\}$, if the $m^{th}$ edge in the set is defined by $(i,j)$, then the $(m+1)^{th}$ edge is defined by $(j,l)$. We define the \textit{length} of a path as the number of edges in the path. If for each pair of vertices $p$ and $q$, such that $p\in\{1,\ldots,n\},q\in\{1,\ldots,n\}\setminus\{p\}$, in a given digraph there exists a path which begins at $p$ and ends at $q$, we say that the digraph is \textit{strongly connected}. It is the case that for any $A\in M_n$, $A$ is irreducible if and only if the digraph of $A$ is strongly connected \cite{brualdi}. We define a \textit{cycle} to be a path which begins and ends at the same point, and which only intersects itself at this point. We refer to a cycle of length $1$ as a \textit{loop}. Also note that a permutation similarity which swaps the $i^{th}$ and $j^{th}$ rows/columns of $A$ is reflected in the digraph of $A$ by swapping the labels of the $i^{th}$ and $j^{th}$ vertices of the digraph.

The \textit{circumference} of a digraph $G$ is defined as the length of the longest cycle present within the graph. We write this as $circ(G)$. Note that as the circumference decreases, the minimum number of edges needed to be strongly connected increases. The following theorem gives the minimum number of edges of a digraph $G$ on $n$ vertices given that $circ(G)=k$. 
\begin{theorem}
 Let $k,n$ be integers such that $2\leq k\leq n$ and  $n=a(k-1)+b$ 
 for some $a>0$ and $0\leq b< k-1$, 
 define $e_{n,2}=2(n-1)$ and for $k>2$, 
 define 
\[e_{n,k}=\left\{\begin{array}{ll}
ka-1 & \mbox{ if } b=0, (a\geq 2)\\
ka & \mbox{ if } b=1\\
ka+b & \mbox{ if } b>1\\
\end{array}\right..\]
If $G$ is a strongly connected digraph with $n$ vertices,
$circ(G)=k$, then $|E|\geq e_{n,k}$. Moreover, the bound is best possible,
i.e., there is a graph $G_0$ with $n$ vertices, $circ(G_0) = k$ and 
$e_{n,k}$ edges.
\end{theorem}
\noindent\textit{Proof:} Let $G$ be a strongly connected digraph with vertex set $V_n=\{1,\ldots,n\}$, edge set $E$ and $2\leq circ(G)=k$. \\
\textbf{Case 1:} Suppose $k=2$. We proceed by induction on $n$. When $n=2$, we have $(1,2),(2,1)\in E$, and hence $|E|\geq e_{2,2}=2$. Suppose $n\geq 3$ and any strongly connected graph $\bar{G}(V_{n-1},\bar{E})$ with $circ(G)=2$ satisfies $|\bar{E}|\geq e_{n-1,2}= 2(n-2)$. Assume that $|E|< e_{n,2}=2(n-1)$. Note that we can relabel the vertices so that $(n-1,n),(n,n-1)\in E$. Now, \[E\subseteq V_n\times V_n=\underbrace{(V_{n-1}\times V_{n-1})}_{S_1} \cup\underbrace{(V_{n-2}\times \{n\})}_{S_2} \cup \underbrace{(\{n\} \times V_{n-2})}_{S_3} \cup \underbrace{\{(n-1,n),(n,n-1),(n,n)\}}_{S_4} \]
which is a disjoint union of sets. Thus 
\begin{equation}
|E|=|E\cap S_1|+|E\cap S_2|+|E\cap S_3| +|E\cap S_4|< 2(n-1)
\end{equation}
Since $(n-1,n),(n,n-1)\in E$, we also have $|E\cap S_4|\geq 2$ and so 
\[|E\cap S_1|+|E\cap S_2|+|E\cap S_3|<2(n-2)\]
Now, define the edge set $\bar{E}\subseteq V_{n-1}\times V_{n-1}$ as follows 
\[\bar{E}=\underbrace{(E\cap S_1)}_{T_1}\cup \underbrace{\{(j,n-1) \ | \ (j,n)\in E\cap S_2\}}_{T_2} \cup \underbrace{ \{(n-1,j) \ | \ (n,j)\in E\cap S_3\}}_{T_3}\]
That is, we obtain $\bar{E}$ by removing the edges $(n-1,n),(n,n-1),(n,n)$ from $E$ and morphing vertices $n$ and $n-1$ into one vertex, labeling it as $n-1$. Thus, 
\[|\bar{E}| \leq |E\cap S_1|+|E\cap S_2|+|E\cap S_3|<2(n-2)\]
It is easy to verify that if $i,j\in V_{n-1}$ and there is a path from $i$ to $j$ in $G$, then there is a path from $i$ to $j$ in $\bar{G}$. Thus, $\bar{G}$ is strongly connected. Also, if there is a cycle of length $k$ in $\bar{G}$, then there is a cycle of length greater than or equal to $k$ in $G$. Thus,  $circ(G)=2$. This contradicts the induction hypothesis. By mathematical induction, $|E|\geq e_{n,2}$.\medskip\\
\textbf{Case 2:}
Next, assume $3\leq k\leq n$ and $n=a(k-1)+b$. We will prove the theorem by induction on $a$. We start with the following base cases for the (i) $b=0$, that is $a=2$ and $n=2(k-1)$; (ii) $b=1$, that is, $a=1$ and $n=k$; and (iii) $b>1$, that is $a=1$ and $n=k+b-1$.
\begin{enumerate}
\item[(i)] Let $n=2(k-1)$. That is, $a=2$ and $b=0$. Then $e_{n,k}=2k-1=n+1$. We are assuming $G$ is strongly connected and $\mbox{circ}(G)= k<n$. We can relabel the vertices so that there is a $k$-cycle formed by vertices $V_n-V_{n-k}$, consisting of $k$ edges. Additionally, there must an outgoing edge from each vertex $j\in V_{n-k}$. This gives us additional $n-k$ distinct edges. Finally, there must be an outgoing edge from a vertex of $V_n-V_{n-k}$ going to a vertex in $V_{n-k}$. Thus $|E|\geq k+n-k+1=n+1=e_{n,k}$.  
\item[(ii)] Let $n=k$. That is $a=1$ and $b=1$ and $e_{n,k}=k$. It is clear that $|E|\geq k=e_{n,k}$ since there must be an outgoing (equivalently, incoming) edge for each vertex. 
\item[(iii)] Let $n=k+b-1$ for some $b>1$. That is $a=1$ and $e_{n,k}=k+b=n+1$. Using the same argument for $n=2(k-1)$, we get that $|E|\geq n+1=e_{n,k}$.
\end{enumerate} 
Assume that $a\geq 2$ when $b>0$ and $a\geq 3$ when $b=0$. Suppose further that any strongly connected graph $\bar{G}=(V_{n-k+1},\bar{E})$ with $circ(\bar{G})=k$ satisfies $|\bar{E}|\geq e_{n-k+1,k}$. Suppose $|E|< e_{n,k}$. We can relabel the vertices so that $\{n-k+1,\ldots, n\}$ form a $k$-cycle in $G$, where $k<n$. We will define the digraph $\hat{G}$ with vertex set $V_k$ and edge set $\hat{E}=S_1\cup S_2\cup S_3$, where 
\[S_1=E\cap \Big(V_{n-k+1}\times V_{n-k+1}\Big)\]
\[S_2=\Big\{(j,n-k+1)\ |\ (j,s)\in E\cap \Big(V_{n-k}\times (V_n-V_{n-k+1}) \Big)\Big\}\]
\[S_3=\Big\{(n-k+1,j)\ |\ (j,s)\in E\cap \Big((V_n-V_{n-k+1})\times V_{n-k}\Big) \Big\}\]
That is, we remove the edges contained in the $k$-cycle and collapse vertices $n-k+1,\ldots, n$ into one vertex labeled by $n-k+1$. Then $|\hat{E}|\leq |E|-k<e_{n,k}-k=e_{n-k+1,k}$. Note that $\hat{G}$ is strongly connected and $circ(\hat{G})\leq k$. Note that from $\hat{G}$, we can define a strongly connected digraph $\bar{G}$ with $circ(G)=k$ by rearranging its edges, relocating and realigning the edges if necessary without removing or introducing a new edge. This contradicts the induction hypothesis. By mathematical induction, $|E|\geq e_{n,k}$.
\\
For the last assertion, consider $G_0$ to be the digraph on $V_n$ constructed as follows. For $k = 2$, let the edge set of $G_0$ be $E=\{(i,i+1),(i+1,i)\ | \ i=1,\ldots, n-1\}$. 
\\
For $k > 2$, construct a $k$-cycle $1\rightarrow 2 \rightarrow\cdots \rightarrow k\rightarrow 1$;
if there are at least $k-1$ vertices left, construct another cycle $k\rightarrow k+1\rightarrow  \cdots\rightarrow 2k-1\rightarrow k$;
if there are at least $k-1$ vertices construct another cycle $2k-1 \rightarrow 2k \rightarrow \cdots
\rightarrow 3k-2\rightarrow 2k-1$, until we have either $k-2$ vertices left (when $b=0$) or $b-1$ vertices left with $1 \le b  < k-1$ vertex. For the former case or if $b> 0$, use a vertex in the last $k$-cycle and the remaining vertices
to form either a $k-1$-cycle or a $b$-cycle.
 \hfill $\square$

Suppose $G$ is a strongly connected digraph with $n$ vertices, edge set $E$, circumference $k$ and contains $m$ loops. Note that removing the loops does not change the strong connectivity of $G$. It follows from the preceding theorem that 
$|E|-m \geq e_{n,k}$.

\subsection{Minors}

The following lemma is from elementary algebra and it is useful for better defining the properties of the characteristic polynomial of a stable matrix:
\begin{lemma}\label{deg2} Let $p(x)=x^2+cx+d$ be a quadratic polynomial with  real valued coefficients $c,d$. Then $p$ has roots $\lambda_1 ,\lambda_2$ with  ${\rm Re}(\lambda_1)<0$ and ${\rm Re}(\lambda_2)<0$ if and only if $c>0$ and $d>0$.
\end{lemma}
%\textit{Proof:} Since $p$ has real coefficients, then $p(x)$ has either two real roots, or a pair of complex conjugate roots. \begin{itemize}
%\item \textit{Case 1:} Suppose there exists $a_1,a_2\in \mathbb{R}$ such that $x^2+cx+d=(x+a_1)(x+a_2)=x^2+(a_1+a_2)x+a_1a_2$. Then $\lambda_1=-a_1$ and $\lambda_2=-a_2$. So $Re(\lambda_1)<0$ and $Re(\lambda_2)<0$ if and only if $a_1>0$ and $a_2>0$ respectively. By equating coefficients, we have $c=a_1+a_2$ and $d=a_1a_2$. \begin{itemize}
%\item If $a_1,a_2>0$, then $c=a_1+a_2>0$ and $d=a_1a_2>0$. If $c=a_1+a_2>0$ and $d=a_1a_2>0$, then $a_1,a_2>0$.
%\item If $a_1>0,a_2\leq 0$ or $a_1\leq 0,a_2>0$, then $d=a_1a_2\leq 0$. If $d=a_1a_2<0$, then $a_1>0,a_2\leq 0$ or $a_1\leq 0,a_2>0$.
%\item If $a_1,a_2\leq 0$, then $c=a_1+a_2\leq 0$ and $d=a_1a_2\geq 0$. If $c=a_1+a_2\leq 0$ and $d=a_1a_2\geq 0$, then $a_1,a_2\leq 0$.
%\end{itemize}

%\item \textit{Case 2:} Suppose there exists $a,b\in\mathbb{R}$ such that $x^2+cx+d=(x+(a+bi))(x+(a-bi))=x^2+(2a)x+(a^2+b^2)$. Then $Re(\lambda_1),Re(\lambda_2)=-a$. So $Re(\lambda_1)<0$ and $Re(\lambda_2)<0$ if and only if $a>0$. By equating coefficients, we have $c=2a$ and $d=a^2+b^2$.
%\begin{itemize}
%\item If $a>0$, then $c=2a>0$ and $d=a^2+b^2>0$. If $c=2a>0$ and $d=a^2+b^2>0$, then $a>0$.
%\item If $a\leq 0$, then $c=2a\leq 0$. If $c=2a\leq 0$, then $a\leq 0$.
%\end{itemize}
%\end{itemize}
%Therefore, $Re(\lambda_1),Re(\lambda_2)<0$ if and only if $c,d>0$. $\blacksquare$ \medskip\\
An $m\times m$ \textit{principal submatrix} of $A$ is a matrix $B=[b_{ij}]\in M_m$, $1\leq m\leq n$ such that $b_{ij}=a_{v_iv_j}$ for some $v_1<\ldots<v_m\in\{1,\ldots,n\}$. A \textit{principal minor} of $A$ is defined as the determinant of some principal submatrix $B=[b_{ij}]$ of $A$. We denote the $m\times m$ principal minor of $A$ indexed by $v_1<\ldots<v_m\in\{1,\ldots,n\}$ as $M(A)_{v_1,\ldots,v_m}$. For example,
$$
A=\begin{bmatrix}
1 & 1 & 0 \\
1 & 0 & 1 \\
0 & 1 & 0
\end{bmatrix} \quad \Longrightarrow \quad
M(A)_{23}=\det\left(
\begin{bmatrix}
0 & 1 \\
1 & 0
\end{bmatrix}\right),\quad M(A)_{13}=\det\left(
\begin{bmatrix}
1 & 0 \\
0 & 0
\end{bmatrix}\right).
$$
There is a direct relationship 
between the minors of a matrix and its eigenvalues. The sum of all $k\times k$ 
principal minors of a matrix $A$ is equal to the sum of all products of unique 
combinations of $k$ eigenvalues of $A$. That is,
\begin{equation}\label{E_k}
E_k=\sum_{1\leq v_1< \ldots< v_k\leq n} M(A)_{v_1,\ldots,v_k}=\sum_{1\leq u_1< \ldots< u_k\leq n} \lambda_{u_1}\cdots\lambda_{u_k}.
\end{equation}
Furthermore, the coefficient of $t^{n-k}$ in the characteristic polynomial 
$P_A(t)=\det(tI-A)$ of the $A$ is equal to $(-1)^k E_k$. Due to the relationship 
between the minors and the eigenvalues of a matrix, we have the following lemma, 
which is well known:
\begin{lemma}\label{lem6}
If $A\in M_n(\mathbb{R})$ is stable, then the following are true:
\begin{enumerate}
\item For all $k=1,\ldots, n$, the sign of the sum of the $k\times k$ minors of $A$, $E_k$, is $(-1)^k$.
\item The characteristic polynomial of $A$, $$P_A(t)=\det(tI-A)=\sum_{k=0}^n(-1)^kE_kt^{n-k},$$ has all positive coefficients.
\end{enumerate}
\end{lemma}
%\begin{proof} Since $A\in M_n(\mathbb{R})$, the eigenvalues of $A$ are made up of $n/2$ conjugate pairs if $n$ is even, or $(n-1)/2$ conjugate pairs and one real valued eigenvalue if $n$ is odd. Therefore the characteristic polynomial of $A$ can be factored into several second degree polynomials, as well as one first degree polynomial if $n$ is odd. From Lemma \ref{deg2}, a second degree polynomial with real coefficients has roots with negative real parts if and only if each coefficient of the polynomial is positive. Similarly, a first degree polynomial with real coefficients has a negative root if and only if the coefficient is positive. Since each of the roots of the characteristic polynomial have negative real parts, each factor of the polynomial must have positive coefficients. Therefore the expanded polynomial must have only positive coefficients.\end{proof}

Note that the above lemma gives us a necessary condition for a given matrix $A$ to be stable. This condition will be very important in our work. If a given sign pattern can be realized by a real valued matrix $A$ which meets the condition that  the sign of the sum of the $k\times k$ minors of $A$ is $(-1)^k$, then we say that this sign pattern has \textit{correct minors}. If for some $k$, the sum of $k\times k$ minors is equal to zero, then that sign pattern cannot be PS, as this would imply that either some of its eigenvalues are positive and some are negative, or that at least one of the eigenvalues is equal to zero.%\medskip\\

The condition on the coefficients of $P_A(t)$ is necessary for the stability of $A$, however it is not sufficient. For example if 
$A = {\small 
\begin{bmatrix} - 0.8 & - 0.81 & -1.01 \cr 1 & 0 & 0 \cr 0 & 1 & 0 \cr
\end{bmatrix}}$, then 
$$P_A(t)=t^3+0.8t^2+0.81t+1.01=(t+1)(t-0.1+i)(t-0.1-i).$$
So $P_A(t)$ has positive coefficients,
but $A$ has eigenvalues $\lambda=0.1\pm i$ which have strictly positive real parts, 
and so $A$ is not stable. 
%\textcolor{red}{In section 2.4, we will introduce a 
%necessary and sufficient condition for the stability of $A$.}

\subsection{Digraph Cycles}

There is a direct relationship between the minors of a matrix and the cycles present in its 
digraph. If two or more cycles do not share any vertices, then we say that they are 
\textit{independent}. If the digraph of a sign pattern contains a cycle made up of 
$k$ edges, then this implies that at least one of its $k\times k$ minors is not equal 
to zero. Additionally, if there exist independent cycles of length $a_1,\ldots,a_m$,
then this implies that, if $\sum_{i=1}^{m}a_i\leq n$, at least one of its 
$(\sum_{i=1}^{m}a_i)\times(\sum_{i=1}^{m}a_i)$ minors is not equal to zero. Below are 
examples of a digraph with correct minors, and one without \\
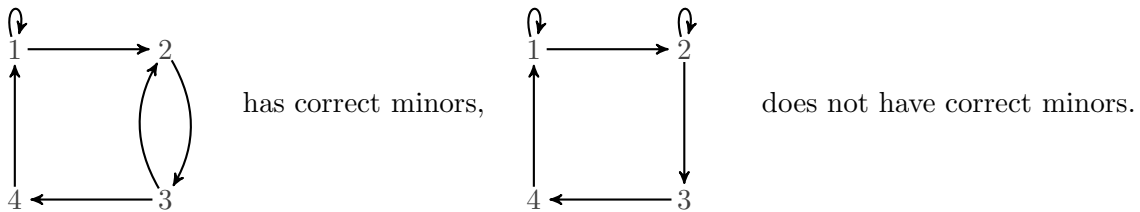
$$
\begin{minipage}{0.2\textwidth}\begin{tikzpicture}[->,>=stealth',shorten >=1pt,auto,node distance=2cm,thick,main node/.style={circle, fill opacity=0.75,minimum size = 6pt, inner sep = 0pt}]

  \node[main node] (1) {1};
  \node[main node] (2) [right of=1] {2};
  \node[main node] (3) [below of=2] {3};
  \node[main node] (4) [left of=3] {4};

  \path
    (1) edge [loop above] node {} (1)
        edge node {} (2)
    (2) edge [bend left] node {} (3)
    (3) edge node {} (4)
        edge [bend left] node {} (2)
    (4) edge node {} (1);

\end{tikzpicture}\end{minipage} \mbox{has correct minors,} \quad
\begin{minipage}{0.2\textwidth}\begin{tikzpicture}[->,>=stealth',shorten >=1pt,auto,node distance=2cm,thick,main node/.style={circle, fill opacity=0.75,minimum size = 6pt, inner sep = 0pt}]

  \node[main node] (1) {1};
  \node[main node] (2) [right of=1] {2};
  \node[main node] (3) [below of=2] {3};
  \node[main node] (4) [left of=3] {4};

  \path
    (1) edge [loop above] node {} (1)
        edge node {} (2)
    (2) edge [loop above] node {} (2)
    	edge node {} (3)
    (3) edge node {} (4)
    (4) edge node {} (1);

\end{tikzpicture}\end{minipage} \mbox{does not have correct minors.}
$$
Therefore, if a given digraph has independent cycles whose lengths add up to $1,\ldots,n$, then we can assign signs to the entries of the corresponding matrix such that it has correct minors.

\section{Candidate Digraphs}
%\begin{center}
%$\textbf{Irreducible $n=7$ digraphs with 10 nonzero %entries}
%\end{center}
In this section we  construct  all candidate digraphs with $7$ vertices and $10$ edges which allows correct minors for a stable matrix. In order to better organize this list, we classify the graphs based on its circumference (the maximum length of cycle in the graph).

\medskip
%\noindent\begin{minipage}{0.6\textwidth}
\noindent\textbf{Case 1:} $circ(G)=7$.

\noindent In this case there must be at least one loop (see the minimum configuration below), and either there are at least two loops or there is exactly one loop and a 2-cycle.
%\end{minipage}
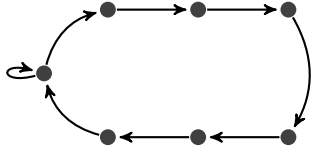
\begin{center}
\begin{minipage}{0.4\textwidth}
\begin{tikzpicture}[->, >=stealth',shorten >=1pt,auto,node distance=1.2cm,thick,main node/.style={circle,fill, fill opacity=0.75,minimum size = 6pt, inner sep = 1pt}]
  \node[main node] (1) {};
  \node[main node] (2) [above right of=1] {};
  \node[main node] (7) [below right of=1] {};
  \node[main node] (3) [right of=2] {};
  \node[main node] (6) [right of=7] {};
  \node[main node] (4) [right of=3] {};
  \node[main node] (5) [right of=6] {};

  \path
    (1) edge [loop left] node {} (1)
        edge[bend left] node[above] {} (2)
    (2) edge node[above] {} (3)
	(3) edge node[above] {} (4)
	(4) edge[bend left] node[right] {} (5)
	(5) edge node[below] {} (6)
	(6) edge node[below] {} (7)
	(7) edge[bend left] node[below] {} (1);
\end{tikzpicture}
\end{minipage}
\end{center}
\medskip

\noindent\textbf{Case 1.1}: There are at least two loops. Then $9$ edges have been utilized. Suppose another edge is added to create a $k$-cycle where $k<7$. The possible sizes of nonzero minors are $1,2,k,k+1,k+2,7$ (possibly less if the $k$ cycle intersects any of the two loops.) Thus, there is at least one $3<r<7$ such that the $r\times r$ minor of the adjacency matrix is zero. Therefore $G$ is not potentially stable.\\
\textbf{Case 1.2}: There is exactly one loop and a $2$-cycle of two adjacent (numbering-wise) vertices. This utilizes $9$ edges. Suppose the remaining edge is contained in a $k$-cycle, where $2\leq k <7$. If the $2$-cycle and the loop have a vertex in common, then the possible sizes of nonzero minors are $1,2,k,k+1,k+2,7$, so we miss at least one minor, and therefore $G$ is not potentially stable. Similarly, if the $k$-cycle has a vertex in common with either the loop or the $2$-cycle, we get a non-PS adjacency matrix. Thus, the $2$-cycle, loop and $k$ cycle must be pairwise disjoint. In this case, the possible sizes of nonzero minors are $1,2,3,k,k+1,k+2,k+3,7$. Thus, $k=4$ or $k=3$. \\
In this case we have the candidate graphs as shown in Figure \ref{fig5}.1, Figure \ref{fig5}.2, Figure \ref{fig5}.3 and Figure \ref{fig5}.4.\\
\textbf{Case 1.3}: There is exactly one loop and a $2$-cycle of two non-adjacent (numbering-wise) vertices. Suppose these two additional edges create a $k$-cycle and an $r$-cycle and nonzero minors of sizes $1,2,3,k,r,r+1,7$. Thus $(k,r)=(4,5)$.\\
In this case we have the candidate graph Figure \ref{fig5}.5.
\medskip

\noindent\textbf{Case 2:} $circ(G) = 6$.

\noindent In this case there must be at least on loop. Either there is a loop on the vertex that does not belong to the $6$-cycle or there is none (see the two possible configurations below.) Two of the three edges must be utilized to make sure that the graph is strongly connected. That is, one edge must be coming from the lone vertex and one must be going to the lone vertex.

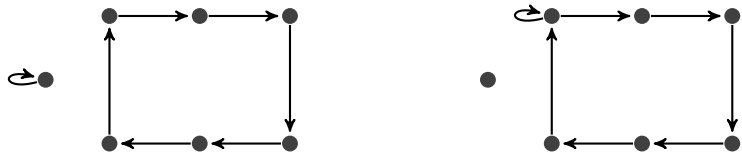
\begin{center}
\begin{tikzpicture}[->, >=stealth',shorten >=1pt,auto,node distance=1.2cm,thick,main node/.style={circle,fill, fill opacity=0.75,minimum size = 6pt, inner sep = 1pt}]
  \node[main node] (1) {};
  \node[main node] (2) [above right of=1] {};
  \node[main node] (7) [below right of=1] {};
  \node[main node] (3) [right of=2] {};
  \node[main node] (6) [right of=7] {};
  \node[main node] (4) [right of=3] {};
  \node[main node] (5) [right of=6] {};

  \path
    (1) edge [loop left] node {} (1)
    (2) edge node[above] {} (3)
	(3) edge node[above] {} (4)
	(4) edge node[right] {} (5)
	(5) edge node[below] {} (6)
	(6) edge node[below] {} (7)
	(7) edge node[below] {} (2);
\end{tikzpicture} \hspace{2cm}
\begin{tikzpicture}[->, >=stealth',shorten >=1pt,auto,node distance=1.2cm,thick,main node/.style={circle,fill, fill opacity=0.75,minimum size = 6pt, inner sep = 1pt}]
  \node[main node] (1) {};
  \node[main node] (2) [above right of=1] {};
  \node[main node] (7) [below right of=1] {};
  \node[main node] (3) [right of=2] {};
  \node[main node] (6) [right of=7] {};
  \node[main node] (4) [right of=3] {};
  \node[main node] (5) [right of=6] {};

  \path
    (2) edge[loop left] node[above] {} (2)
    	edge node[above] {} (3)
	(3) edge node[above] {} (4)
	(4) edge node[right] {} (5)
	(5) edge node[below] {} (6)
	(6) edge node[below] {} (7)
	(7) edge node[below] {} (2);
\end{tikzpicture}
\end{center}

\noindent \textbf{Case 2.1:} There is a loop in the lone vertex (say $v_1$) and another loop in another vertex. So far, we can guarantee nonzero minors of size $1,2,6,7$. For the two remaining edges, one must be outgoing from $v_1$ and one must be incoming from $v_1$. If these two edges form a $k$-cycle (which intersects a loop and the $6$-cycle), then we get nonzero minors of size $k$ and $k+1$ and nothing else. Thus $G$ will not be potentially stable. \\
\textbf{Case 2.2:} There is a loop in the lone vertex and no loop in any other vertex. Suppose the outgoing and incoming edge to the lone vertex form a $k$-cycle (which intersects the loop and the 6-cycle), with $k<7$. Then minors of size $1,k,6,7$ are nonzero. Suppose the remaining edge gives rise to another cycle of length $1<r<7$ (this means it must necessarily intersect the $6$-cycle.  This may give rise to nonzero minors of size $r,r+1$ and $r+k$ (less if the $r$-cycle also intersects with the loop or the $k$-cycle). Thus, the $r$-cycle must not intersect with the $k$-cycle and $\{k,r,r+1,r+k\}=\{2,3,4,5\}$. There is no choice but for $k=2$ and $r=3$.\\
Thus, in this case, we have the candidate graphs as in  Figure \ref{fig5}.6 and Figure \ref{fig5}.7.\\
\textbf{Case 2.3}: There is no loop in the lone vertex. Hence, there is at least one loop intersecting the $6$-cycle. Suppose the incoming and outgoing edges to the lone vertex forms a $k$-cycle, where $1<k<7$. At this point, $9$ edges have been accounted for and nonzero minors of sizes $1,6,k,k+1$. Suppose the last edge forms gives rise to an $r$-cycle, where $r<7$ that, by assumption, must intersect the $6$-cycle. If this $r$-cycle intersects the loop or the $k$-cycle, then there will be a zero minor and thus, the adjacency matrix cannot be PS. If the $r$-cycle, loop and the $k$-cycle are pairwise disjoint, then
we get additional nonzero minors of size $r,r+1,r+k,r+k+1$. We want $\{2,3,4,5,7\}\subseteq\{k,k+1,r,r+1,r+k,r+k+1\}$. Thus, either $(r,k)=(2,4)$ or $(r,k)=(4,2)$.\\
Thus, in this case, we have the candidate graphs: Figure \ref{fig5}.8 and Figure \ref{fig5}.9.
\medskip

\noindent\textbf{Case 3:} $circ(G) = 5$.

\noindent In this case a $5$-cycle uses $5$ edges, and another edge must form a loop. At least three out of the four remaining edges must be used to ensure strong connectedness of the graph. Either there is a loop intersecting the $5$-cycle or there is none (see  the two possible configurations below.)

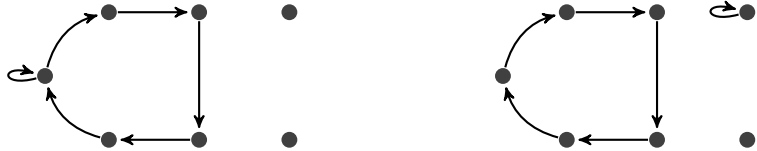
\begin{center}
\begin{tikzpicture}[->, >=stealth',shorten >=1pt,auto,node distance=1.2cm,thick,main node/.style={circle,fill, fill opacity=0.75,minimum size = 6pt, inner sep = 1pt}]
  \node[main node] (1) {};
  \node[main node] (2) [above right of=1] {};
  \node[main node] (7) [below right of=1] {};
  \node[main node] (3) [right of=2] {};
  \node[main node] (6) [right of=7] {};
  \node[main node] (4) [right of=3] {};
  \node[main node] (5) [right of=6] {};

  \path
    (1) edge [loop left] node {} (1)
        edge[bend left] node[above] {} (2)
    (2) edge node[above] {} (3)
	(3) edge node[above] {} (6)
	(6) edge node[below] {} (7)
	(7) edge[bend left] node[below] {} (1);
\end{tikzpicture}\hspace{2.5cm}
\begin{tikzpicture}[->, >=stealth',shorten >=1pt,auto,node distance=1.2cm,thick,main node/.style={circle,fill, fill opacity=0.75,minimum size = 6pt, inner sep = 1pt}]
  \node[main node] (1) {};
  \node[main node] (2) [above right of=1] {};
  \node[main node] (7) [below right of=1] {};
  \node[main node] (3) [right of=2] {};
  \node[main node] (6) [right of=7] {};
  \node[main node] (4) [right of=3] {};
  \node[main node] (5) [right of=6] {};

  \path
    (1) edge[bend left] node[above] {} (2)
    (2) edge node[above] {} (3)
	(3) edge node[above] {} (6)
	(6) edge node[below] {} (7)
	(7) edge[bend left] node[below] {} (1)
	(4) edge [loop left] node {} (4);
\end{tikzpicture}
\end{center}

\noindent \textbf{Case 3.1:} There is a loop intersecting the $5$-cycle (the one we choose at the beginning, there may be more than one $5$-cycle).  Either there is an edge between the two remaining vertices or there is none.\\
\textbf{Subcase 3.1.1:} Suppose there is no edge connecting the the two vertices (let's call them $v_1$ and $v_2$). Then, two of the four remaining edges should connect $v_1$ to vertices in the $5$-cycle to form a $k$-cycle, where $k\leq 5$. Similarly, the remaining two edges must connect $v_2$ to vertices in the $5$-cycle to form an $r$-cycle, where $k\leq 5$. Assuming the $k$-cycle, $r$-cycle and the loop are disjoint, then we have nonzero minors of size $1,5,k,r,k+1,r+1,k+r,k+r+1$. (Note that if they are not pairwise disjoint, there will be at least minor size that will be missing.)  Thus $\{2,3,4,6,7\}\in \{k,r,k+1,r+1,k+r,k+r+1\}$. Thus $(k,r)=(2,4)$ or $(k,r)=(4,2)$.\\
Thus, we have the candidate graph: Figure \ref{fig5}.10.\\
\textbf{Subcase 3.1.2:} Suppose $v_1$ and $v_2$ form a $2$-cycle and $v_2$ is not adjacent to any vertex in the 5-cycle. So far, we have accounted for 8 edges and nonzero minors of sizes $1,2,3,5$. The two remaining edges must be incoming and outgoing from $v_1$ to make a strongly connected graph. Say these two remaining edges forms a $k$-cycle, where $k\leq 5$. This adds nonzero minors of size $k$ and $k+1$, which is not enough to make a potentially stable adjacency matrix.\\
\textbf{Subcase 3.1.3:} Suppose $v_1$ and $v_2$ are part of a $k$-cycle, with $2<k\leq 5$. So far, we have accounted for at least $9$ edges and nonzero minors of sizes $1,k,5,k+1$, where $\geq 3$. To get a nonzero $2\times 2$ minor, either there must be another loop or there is a two cycle.  Adding a loop can only guarantee at least $2$ more nonzero minor sizes. Thus, there must be a $2$-cycle in the graph. If the two cycle is disjoint from the $5$-cycle (that is, $v_1$ and $v_2$ form the $2$-cycle), we only get nonzero minors of size $\{1,2,3,5,k,k+1,7\}\neq \{1,2,3,4,5,6,7\}$. Hence the two vertices in the $2$-cycle must be part of the $5$-cycle. In this case, we get nonzero minors, $1,2,3,k,k+1,k+2,k+3,5$. Thus $k=4$.\\
Thus, we have the candidate graph: Figure \ref{fig5}.11.

\smallskip
\noindent \textbf{Case 3.2:} There is no loop intersecting the $5$-cycle. Again, either there is an edge connecting the remaining two vertices $v_1$ and $v_2$ or there is none.  \\
\textbf{Subcase 3.2.1:} There is no edge connecting the remaining two vertices $v_1$ and $v_2$. Then, two of the four remaining edges should connect $v_1$ to vertices in the $5$-cycle to form a $k$-cycle, where $k\leq 5$. Similarly, the remaining two edges must connect $v_2$ to vertices in the $5$-cycle to form an $r$-cycle, where $k\leq 5$. Assuming the $k$-cycle, $r$-cycle and the loop are disjoint, then we have nonzero minors of size $1,5,6,k,r,r+1,r+k$. Note that there is no choice of $2\leq k,r\leq 5$ that will give a complete set of nonzero minors. Thus, this case will not give a PS adjacency matrix.\\
\textbf{Subcase 3.2.2:} Suppose $v_1$ and $v_2$ form a $2$-cycle and one of $v_1$ or $v_2$ is not adjacent to any vertex in the 5-cycle. So far, we have accounted for 8 edges and nonzero minors of sizes $1,2,5,6,7$. The two remaining edges must be incoming and outgoing from $v_1,v_2$ to make a strongly connected graph. Say these two remaining edges forms a $k$-cycle, where $k\leq 5$. This adds nonzero minors of size $k$ and possibly (if the $k$-cycle does not contain the loop) $k+1$.\\
Thus, we have the following graph: Figure \ref{fig5}.12.\\
\textbf{Subcase 3.2.3:} Suppose $v_1$ and $v_2$ are part of a $k$-cycle, with $2<k\leq 5$. So far, we have accounted for at least $9$ edges and nonzero minors of sizes $1,5,6,k$, where $k\geq 3$. To get a nonzero $2\times 2$ minor, there should be another loop or a 2-cycle. Adding a loop can only guarantee at most two more sizes of nonzero minors. If there is a $2$-cycle between $v_1$ and $v_2$, we get additional nonzero minor sizes $2,7$, which is not enough for the graph to be PS. If the $2$-cycle is disjoint with the $k$-cycle, we get additional nonzero minor sizes $2,3,k+2$. This is still not enough to get a PS matrix.
\medskip

\noindent\textbf{Case 4:} $circ(G) = 4$.

\noindent Let $v_1$, $v_2$, $v_3$, $v_4$ form a 4-cycle. For each of $v_5$, $v_6$, $v_7$, there must be incoming edges $\{5,in\},\{6,in\}, \{7,in\}$ and outgoing edges $\{5,out\},\{6,out\}, \{7,out\}$. Note that 
$$\{\{5,in\},\{6,in\}, \{7,in\}\}\cap \{\{5,out\},\{6,out\}, \{7,out\}\}$$ 
must have at least 1 element since we still have to account for the loop. We can list 
down all possible  nonequivalent strongly connected graphs with less than $9$ edges and maximum cycle length 4 as follows:
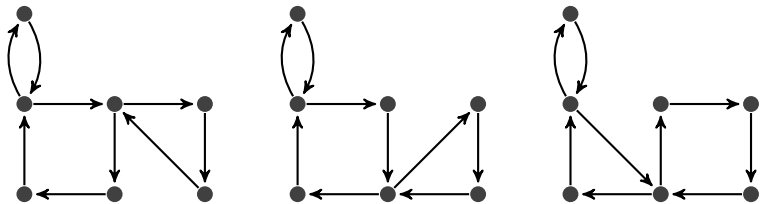
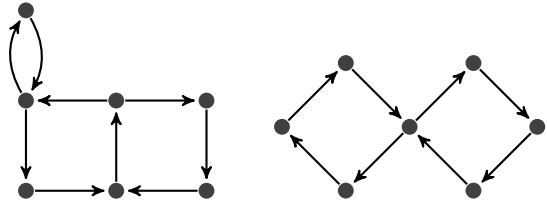
\begin{center}
\begin{tikzpicture}[->, >=stealth',shorten >=1pt,auto,node distance=1.2cm,thick,main node/.style={circle,fill, fill opacity=0.75,minimum size = 6pt, inner sep = 1pt}]
  \node[main node] (1) {};
  \node[main node] (2) [below of =1] {};
  \node[main node] (3) [right of=2] {};
  \node[main node] (4) [right of=3] {};
  \node[main node] (5) [below of=4] {};
  \node[main node] (6) [left of=5] {};
  \node[main node] (7) [left of=6] {};

  \path
    (1) edge[bend left] node[above] {} (2)
    (2) edge[bend left] node[above] {} (1)
    	edge node[above] {} (3)
    (3)	edge node[above] {} (6)
    	edge node[above] {} (4)
    (4)	edge node[above] {} (5)
    (5)	edge node[above] {} (3)	
    (6)	edge node[above] {} (7)
    (7)	edge node[above] {} (2);
\end{tikzpicture}\qquad
\begin{tikzpicture}[->, >=stealth',shorten >=1pt,auto,node distance=1.2cm,thick,main node/.style={circle,fill, fill opacity=0.75,minimum size = 6pt, inner sep = 1pt}]
  \node[main node] (1) {};
  \node[main node] (2) [below of =1] {};
  \node[main node] (3) [right of=2] {};
  \node[main node] (4) [right of=3] {};
  \node[main node] (5) [below of=4] {};
  \node[main node] (6) [left of=5] {};
  \node[main node] (7) [left of=6] {};

  \path
    (1) edge[bend left] node[above] {} (2)
    (2) edge[bend left] node[above] {} (1)
    	edge node[above] {} (3)
    (3)	edge node[above] {} (6)
    (4)	edge node[above] {} (5)
    (5)	edge node[above] {} (6)	
    (6)	edge node[above] {} (7)
    	edge node[above] {} (4)
    (7)	edge node[above] {} (2);
\end{tikzpicture}\qquad
\begin{tikzpicture}[->, >=stealth',shorten >=1pt,auto,node distance=1.2cm,thick,main node/.style={circle,fill, fill opacity=0.75,minimum size = 6pt, inner sep = 1pt}]
  \node[main node] (1) {};
  \node[main node] (2) [below of =1] {};
  \node[main node] (3) [right of=2] {};
  \node[main node] (4) [right of=3] {};
  \node[main node] (5) [below of=4] {};
  \node[main node] (6) [left of=5] {};
  \node[main node] (7) [left of=6] {};

  \path
    (1) edge[bend left] node[above] {} (2)
    (2) edge[bend left] node[above] {} (1)
    	edge node[above] {} (6)
    (3)	edge node[above] {} (4)
    (4)	edge node[above] {} (5)
    (5)	edge node[above] {} (6)	
    (6)	edge node[above] {} (7)
    	edge node[above] {} (3)
    (7)	edge node[above] {} (2);
\end{tikzpicture}\medskip\\
\begin{tikzpicture}[->, >=stealth',shorten >=1pt,auto,node distance=1.2cm,thick,main node/.style={circle,fill, fill opacity=0.75,minimum size = 6pt, inner sep = 1pt}]
  \node[main node] (1) {};
  \node[main node] (2) [below of =1] {};
  \node[main node] (3) [right of=2] {};
  \node[main node] (4) [right of=3] {};
  \node[main node] (5) [below of=4] {};
  \node[main node] (6) [left of=5] {};
  \node[main node] (7) [left of=6] {};

  \path
    (1) edge[bend left] node[above] {} (2)
    (2) edge[bend left] node[above] {} (1)
    	edge node[above] {} (7)
    (3)	edge node[above] {} (4)
    	edge node[above] {} (2)
    (4)	edge node[above] {} (5)
    (5)	edge node[above] {} (6)	
    (6)	edge node[above] {} (3)
    (7)	edge node[above] {} (6);
\end{tikzpicture}\qquad \begin{tikzpicture}[->, >=stealth',shorten >=1pt,auto,node distance=1.2cm,thick,main node/.style={circle,fill, fill opacity=0.75,minimum size = 6pt, inner sep = 1pt}]
  \node[main node] (1) {};
  \node[main node] (2) [above right of=1] {};
  \node[main node] (7) [below right of=1] {};
  \node[main node] (3) [below right of=2] {};
  \node[main node] (4) [above right of=3] {};
  \node[main node] (6) [below right of=3] {};
  \node[main node] (5) [below right of=4] {};

  \path
    (1) edge node[above] {} (2)
    (2) edge node[above] {} (3)
	(3) edge node[above] {} (7)
	(7) edge node[right] {} (1)
	(3) edge node[below] {} (4)
	(4) edge node[below] {} (5)
	(5) edge node[below] {} (6)
	(6) edge node[below] {} (3);
\end{tikzpicture}
\end{center}
For the top left and middle graphs, adding a loop will give an adjacency matrix that has zero determinant. For the top rightmost graph, a loop that is disjoint from the $2$-cycle and $4$-cycle must be  added to get all nonzero minors. For the lower left graph, a loop must be added so that the loop, a $4$-cycle and a $2$-cycle are all disjoint. Finally, for the lower right graph, a loop and a $2$-cycle must be added as shown in the figure below to get nonzero minors.\\
Thus, we have the following candidate graphs:  Figure \ref{fig5}.13, Figure \ref{fig5}.14, and Figure \ref{fig5}.15.
\medskip

\noindent\textbf{Case 5:} $circ(G) = 3$.\\
By our formula, $e_{7,3}=9$, so any digraph with $7$ vertices and a circumference of $3$ must contain at least $9$ edges in order to be strongly connected. Let $v_1$, $v_2$ and $v_3$ form a $3$-cycle. Since the graph needs at least $9$ edges in order to be strongly connected, it can have at most $1$ loop, giving a total of $9+1=10$ edges. Then the graph must have a $2$-cycle as well.\\
\textbf{Case 5.1:} Suppose the $2$-cycle shares an edge with the $3$-cycle ($v_1v_2v_3$). Then, between the $2$-cycle, the $3$-cycle and the loop, we have used $5$ of the $10$ available edges. So there are $5$ edges remaining with which to connect the vertices $v_4$, $v_5$, $v_6$ and $v_7$. Each of these vertices requires at least one incoming edge and one outgoing edge. Since $circ(G)=3$, it would require at least $3$ edges in order to connect any two of the remaining vertices to the original $3$-cycle. From that point, it would require at least an additional $3$ edges in order to connect the remaining two vertices. However there are only $5$ edges available, and thus the $2$-cycle cannot share an edge with the $3$-cycle.\\
\textbf{Case 5.2:} Suppose the $2$-cycle does not share an edge with the $3$-cycle ($v_1v_2v_3$), say the $2$-cycle is ($v_4v_5$) without loss of generality. Then, between the $2$-cycle, the $3$-cycle and the loop, we have used $6$ of the $10$ available edges. So there are $4$ edges remaining with which to connect the vertices $v_6$ and $v_7$ with the cycle ($v_1v_2v_3$) and the cycle ($v_4v_5$). Since the circumference of the graph is $4$, it would require at least $3$ edges in order to connect $v_6$ and $v_7$ to either the $3$-cycle or the $2$-cycle. Then there is at most $1$ edge remaining, which is insufficient to connect the remaining separated cycles. Therefore the $2$-cycle cannot be separate from the $3$-cycle, and so there are no digraphs with $7$ vertices and a circumference of $3$ which have correct minors.

\noindent\textbf{Case 6:} $circ(G) = 2$.\\
 By our formula, $e_{7,2}=12$, so any digraph with $7$ vertices and a circumference of $2$ must contain at least $12$ edges in order to be strongly connected. However we are assuming only $10$ edges, and therefore we cannot have a circumference of $2$.

\noindent Summarizing the above discussion, we reach the main result in this section:
\begin{proposition}\label{pro3}
Suppose that $(V,E)$ is a strongly connected digraph with $7$ vertices and $10$ edges which has all non-zero minors. Then $(V,E)$ is equivalent to one of digraphs in Figure \ref{fig5}.
\end{proposition}

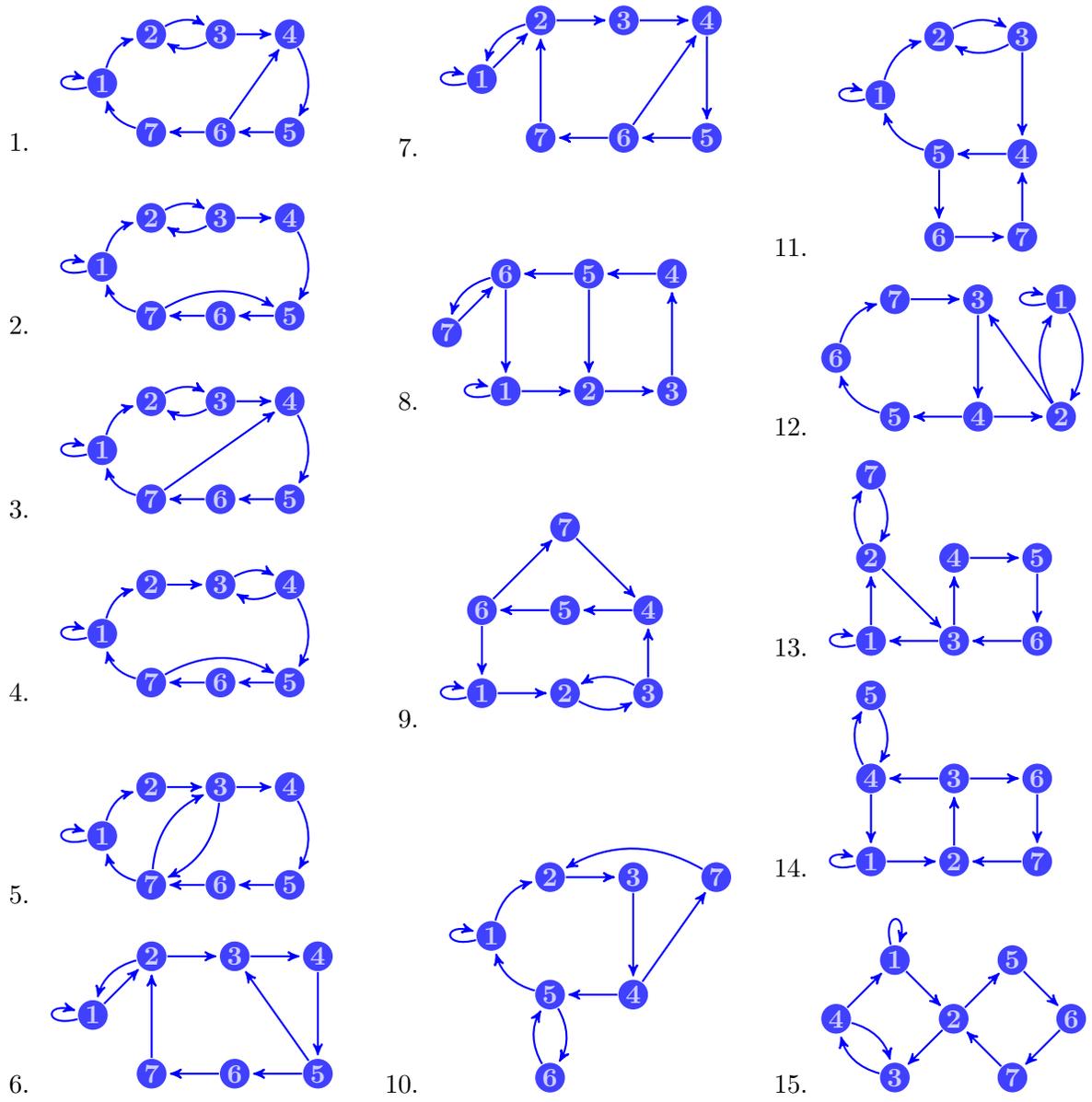
\begin{figure}
\begin{multicols}{3}
\begin{enumerate}
\item \begin{tikzpicture}[->, >=stealth',shorten >=1pt,auto,node distance=1cm,thick,main node/.style={circle,fill, fill opacity=0.75,minimum size = 6pt, inner sep = 1pt}]
  \node[color=blue,main node] (1) {\textcolor{white}{\textbf{1}}};
  \node[color=blue,main node] (2) [above right of=1] {\textcolor{white}{\textbf{2}}};
  \node[color=blue,main node] (7) [below right of=1] {\textcolor{white}{\textbf{7}}};
  \node[color=blue,main node] (3) [right of=2] {\textcolor{white}{\textbf{3}}};
  \node[color=blue,main node] (6) [right of=7] {\textcolor{white}{\textbf{6}}};
  \node[color=blue,main node] (4) [right of=3] {\textcolor{white}{\textbf{4}}};
  \node[color=blue,main node] (5) [right of=6] {\textcolor{white}{\textbf{5}}};

  \path
    (1) edge [color=blue,loop left] node {} (1)
        edge[color=blue,bend left] node[above] {} (2)
    (2) edge[color=blue,bend left] node[above] {} (3)
	(3) edge[color=blue] node[above] {} (4)
	    edge[color=blue,bend left] node[above] {} (2)
	(4) edge[color=blue,bend left] node[right] {} (5)
	(5) edge[color=blue] node[below] {} (6)
	(6) edge[color=blue] node[below] {} (7)
		edge[color=blue] node[right] {} (4)
	(7) edge[color=blue,bend left] node[below] {} (1);
\end{tikzpicture}
\item
\begin{tikzpicture}[->, >=stealth',shorten >=1pt,auto,node distance=1cm,thick,main node/.style={circle,fill, fill opacity=0.75,minimum size = 6pt, inner sep = 1pt}]
  \node[color=blue,main node] (1) {\textcolor{white}{\textbf{1}}};
  \node[color=blue,main node] (2) [above right of=1] {\textcolor{white}{\textbf{2}}};
  \node[color=blue,main node] (7) [below right of=1] {\textcolor{white}{\textbf{7}}};
  \node[color=blue,main node] (3) [right of=2] {\textcolor{white}{\textbf{3}}};
  \node[color=blue,main node] (6) [right of=7] {\textcolor{white}{\textbf{6}}};
  \node[color=blue,main node] (4) [right of=3] {\textcolor{white}{\textbf{4}}};
  \node[color=blue,main node] (5) [right of=6] {\textcolor{white}{\textbf{5}}};

  \path
    (1) edge [color=blue,loop left] node {} (1)
        edge[color=blue,bend left] node[above] {} (2)
    (2) edge[color=blue,bend left] node[above] {} (3)
	(3) edge[color=blue] node[above] {} (4)
	    edge[color=blue,bend left] node[above] {} (2)
	(4) edge[color=blue,bend left] node[right] {} (5)
	(5) edge[color=blue] node[below] {} (6)
	(6) edge[color=blue] node[below] {} (7)
	(7) edge[color=blue,bend left] node[below] {} (1)
		edge[color=blue,bend left] node[right] {} (5);
\end{tikzpicture}
\item
\begin{tikzpicture}[->, >=stealth',shorten >=1pt,auto,node distance=1cm,thick,main node/.style={circle,fill, fill opacity=0.75,minimum size = 6pt, inner sep = 1pt}]
  \node[color=blue,main node] (1) {\textcolor{white}{\textbf{1}}};
  \node[color=blue,main node] (2) [above right of=1] {\textcolor{white}{\textbf{2}}};
  \node[color=blue,main node] (7) [below right of=1] {\textcolor{white}{\textbf{7}}};
  \node[color=blue,main node] (3) [right of=2] {\textcolor{white}{\textbf{3}}};
  \node[color=blue,main node] (6) [right of=7] {\textcolor{white}{\textbf{6}}};
  \node[color=blue,main node] (4) [right of=3] {\textcolor{white}{\textbf{4}}};
  \node[color=blue,main node] (5) [right of=6] {\textcolor{white}{\textbf{5}}};

  \path
    (1) edge [color=blue,loop left] node {} (1)
        edge[color=blue,bend left] node[above] {} (2)
    (2) edge[color=blue,bend left] node[above] {} (3)
	(3) edge[color=blue] node[above] {} (4)
	    edge[color=blue,bend left] node[above] {} (2)
	(4) edge[color=blue,bend left] node[right] {} (5)
	(5) edge[color=blue] node[below] {} (6)
	(6) edge[color=blue] node[below] {} (7)
	(7) edge[color=blue,bend left] node[below] {} (1)
		edge[color=blue] node[right] {} (4);
\end{tikzpicture}
\item
\begin{tikzpicture}[->, >=stealth',shorten >=1pt,auto,node distance=1cm,thick,main node/.style={circle,fill, fill opacity=0.75,minimum size = 6pt, inner sep = 1pt}]
  \node[color=blue,main node] (1) {\textcolor{white}{\textbf{1}}};
  \node[color=blue,main node] (2) [above right of=1] {\textcolor{white}{\textbf{2}}};
  \node[color=blue,main node] (7) [below right of=1] {\textcolor{white}{\textbf{7}}};
  \node[color=blue,main node] (3) [right of=2] {\textcolor{white}{\textbf{3}}};
  \node[color=blue,main node] (6) [right of=7] {\textcolor{white}{\textbf{6}}};
  \node[color=blue,main node] (4) [right of=3] {\textcolor{white}{\textbf{4}}};
  \node[color=blue,main node] (5) [right of=6] {\textcolor{white}{\textbf{5}}};

  \path
    (1) edge [color=blue,loop left] node {} (1)
        edge[color=blue,bend left] node[above] {} (2)
    (2) edge[color=blue] node[above] {} (3)
	(3) edge[color=blue, bend left] node[above] {} (4)
	(4) edge[color=blue,bend left] node[right] {} (5)
		edge[color=blue,bend left] node[above] {} (3)
	(5) edge[color=blue] node[below] {} (6)
	(6) edge[color=blue] node[below] {} (7)
	(7) edge[color=blue,bend left] node[below] {} (1)
		edge[color=blue,bend left] node[right] {} (5);
\end{tikzpicture}\\
\item
\begin{tikzpicture}[->, >=stealth',shorten >=1pt,auto,node distance=1cm,thick,main node/.style={circle,fill, fill opacity=0.75,minimum size = 6pt, inner sep = 1pt}]
  \node[color=blue,main node] (1) {\textcolor{white}{\textbf{1}}};
  \node[color=blue,main node] (2) [above right of=1] {\textcolor{white}{\textbf{2}}};
  \node[color=blue,main node] (7) [below right of=1] {\textcolor{white}{\textbf{7}}};
  \node[color=blue,main node] (3) [right of=2] {\textcolor{white}{\textbf{3}}};
  \node[color=blue,main node] (6) [right of=7] {\textcolor{white}{\textbf{6}}};
  \node[color=blue,main node] (4) [right of=3] {\textcolor{white}{\textbf{4}}};
  \node[color=blue,main node] (5) [right of=6] {\textcolor{white}{\textbf{5}}};

  \path
    (1) edge [color=blue,loop left] node {} (1)
        edge[color=blue,bend left] node[above] {} (2)
    (2) edge[color=blue] node[above] {} (3)
	(3) edge[color=blue] node[above] {} (4)
	 	edge[color=blue,bend left] node[above] {} (7)
	(4) edge[color=blue,bend left] node[right] {} (5)
	(5) edge[color=blue] node[below] {} (6)
	(6) edge[color=blue] node[below] {} (7)
	(7) edge[color=blue,bend left] node[below] {} (1)
		edge[color=blue,bend left] node[below] {} (3);
\end{tikzpicture}
\item
\begin{tikzpicture}[->, >=stealth',shorten >=1pt,auto,node distance=1.2cm,thick,main node/.style={circle,fill, fill opacity=0.75,minimum size = 6pt, inner sep = 1pt}]
    \node[color=blue,main node] (1) {\textcolor{white}{\textbf{1}}};
  \node[color=blue,main node] (2) [above right of=1] {\textcolor{white}{\textbf{2}}};
  \node[color=blue,main node] (7) [below right of=1] {\textcolor{white}{\textbf{7}}};
  \node[color=blue,main node] (3) [right of=2] {\textcolor{white}{\textbf{3}}};
  \node[color=blue,main node] (6) [right of=7] {\textcolor{white}{\textbf{6}}};
  \node[color=blue,main node] (4) [right of=3] {\textcolor{white}{\textbf{4}}};
  \node[color=blue,main node] (5) [right of=6] {\textcolor{white}{\textbf{5}}};

  \path
    (1) edge[color=blue,loop left] node[above] {} (1)
        edge[color=blue] node[above] {} (2)
    (2) edge[color=blue] node[above] {} (3)
    	edge[color=blue,bend right] node[above] {} (1)
	(3) edge[color=blue] node[above] {} (4)
	(4) edge[color=blue] node[right] {} (5)
	(5) edge[color=blue] node[below] {} (6)
		edge[color=blue] node[below] {} (3)
	(6) edge[color=blue] node[below] {} (7)
	(7) edge[color=blue] node[below] {} (2);
\end{tikzpicture}
\item
\begin{tikzpicture}[->, >=stealth',shorten >=1pt,auto,node distance=1.2cm,thick,main node/.style={circle,fill, fill opacity=0.75,minimum size = 6pt, inner sep = 1pt}]
    \node[color=blue,main node] (1) {\textcolor{white}{\textbf{1}}};
  \node[color=blue,main node] (2) [above right of=1] {\textcolor{white}{\textbf{2}}};
  \node[color=blue,main node] (7) [below right of=1] {\textcolor{white}{\textbf{7}}};
  \node[color=blue,main node] (3) [right of=2] {\textcolor{white}{\textbf{3}}};
  \node[color=blue,main node] (6) [right of=7] {\textcolor{white}{\textbf{6}}};
  \node[color=blue,main node] (4) [right of=3] {\textcolor{white}{\textbf{4}}};
  \node[color=blue,main node] (5) [right of=6] {\textcolor{white}{\textbf{5}}};

  \path
    (1) edge[color=blue,loop left] node[above] {} (1)
        edge[color=blue] node[above] {} (2)
    (2) edge[color=blue] node[above] {} (3)
    	edge[color=blue,bend right] node[above] {} (1)
	(3) edge[color=blue] node[above] {} (4)
	(4) edge[color=blue] node[right] {} (5)
	(5) edge[color=blue] node[below] {} (6)
	(6) edge[color=blue] node[below] {} (7)
		edge[color=blue] node[below] {} (4)
	(7) edge[color=blue] node[below] {} (2);
\end{tikzpicture}
\item \begin{tikzpicture}[->, >=stealth',shorten >=1pt,auto,node distance=1.2cm,thick,main node/.style={circle,fill, fill opacity=0.75,minimum size = 6pt, inner sep = 1pt}]
    \node[color=blue,main node] (7) {\textcolor{white}{\textbf{7}}};
  \node[color=blue,main node] (6) [above right of=7] {\textcolor{white}{\textbf{6}}};
  \node[color=blue,main node] (1) [below right of=7] {\textcolor{white}{\textbf{1}}};
  \node[color=blue,main node] (5) [right of=6] {\textcolor{white}{\textbf{5}}};
  \node[color=blue,main node] (2) [right of=1] {\textcolor{white}{\textbf{2}}};
  \node[color=blue,main node] (4) [right of=5] {\textcolor{white}{\textbf{4}}};

  \node[color=blue,main node] (3) [right of=2] {\textcolor{white}{\textbf{3}}};

  \path
  	(1) edge[color=blue] node[below] {} (2)
		edge[color=blue,loop left] node[below] {} (1)
	(2) edge[color=blue] node[below] {} (3)
	(3) edge[color=blue] node[above] {} (4)
	(4) edge[color=blue] node[right] {} (5)
	(5) edge[color=blue] node[below] {} (6)
		edge[color=blue] node[below] {} (2)
	(6) edge[color=blue,bend right] node[below] {} (7)
		edge[color=blue] node[below] {} (1)
  	(7) edge[color=blue] node[above] {} (6);
\end{tikzpicture}
\item
\begin{tikzpicture}[->, >=stealth',shorten >=1pt,auto,node distance=1.2cm,thick,main node/.style={circle,fill, fill opacity=0.75,minimum size = 6pt, inner sep = 1pt}]
  \node[color=blue,main node] (6)  {\textcolor{white}{\textbf{6}}};
  \node[color=blue,main node] (1) [below of=6] {\textcolor{white}{\textbf{1}}};
  \node[color=blue,main node] (5) [right of=6] {\textcolor{white}{\textbf{5}}};
  \node[color=blue,main node] (7) [above of=5] {\textcolor{white}{\textbf{7}}};
  \node[color=blue,main node] (2) [right of=1] {\textcolor{white}{\textbf{2}}};
  \node[color=blue,main node] (4) [right of=5] {\textcolor{white}{\textbf{4}}};

  \node[color=blue,main node] (3) [right of=2] {\textcolor{white}{\textbf{3}}};

  \path
  	(1) edge[color=blue] node[below] {} (2)
		edge[color=blue,loop left] node[below] {} (1)
	(2) edge[color=blue, bend right] node[below] {} (3)
	(3) edge[color=blue] node[above] {} (4)
		edge[color=blue, bend right] node[below] {} (2)
	(4) edge[color=blue] node[right] {} (5)
	(5) edge[color=blue] node[below] {} (6)
	(6) edge[color=blue] node[below] {} (7)
		edge[color=blue] node[below] {} (1)
  	(7) edge[color=blue] node[above] {} (4);
\end{tikzpicture}
\item \begin{tikzpicture}[->, >=stealth',shorten >=1pt,auto,node distance=1.2cm,thick,main node/.style={circle,fill, fill opacity=0.75,minimum size = 6pt, inner sep = 1pt}]
  \node[color=blue,main node] (1) {\textcolor{white}{\textbf{1}}};
  \node[color=blue,main node] (2) [above right of=1] {\textcolor{white}{\textbf{2}}};
  \node[color=blue,main node] (7) [below right of=1] {\textcolor{white}{\textbf{5}}};
  \node[color=blue,main node] (3) [right of=2] {\textcolor{white}{\textbf{3}}};
  \node[color=blue,main node] (6) [right of=7] {\textcolor{white}{\textbf{4}}};
  \node[color=blue,main node] (4) [right of=3] {\textcolor{white}{\textbf{7}}};
  \node[color=blue,main node] (5) [below of=7] {\textcolor{white}{\textbf{6}}};

  \path
    (1) edge[color=blue,loop left] node {} (1)
        edge[color=blue,bend left] node[above] {} (2)
    (2) edge[color=blue] node[above] {} (3)
	(3) edge[color=blue] node[above] {} (6)
	(4) edge[color=blue,bend right] node[above] {} (2)
	(5) edge[color=blue,bend left] node[above] {} (7)		
	(6) edge[color=blue] node[below] {} (7)
	    edge[color=blue] node[below] {} (4)
	(7) edge[color=blue,bend left] node[below] {} (1)
		edge[color=blue,bend left] node[above] {} (5);
\end{tikzpicture}
\item  \begin{tikzpicture}[->, >=stealth',shorten >=1pt,auto,node distance=1.2cm,thick,main node/.style={circle,fill, fill opacity=0.75,minimum size = 6pt, inner sep = 1pt}]
  \node[color=blue,main node] (1) {\textcolor{white}{\textbf{1}}};
  \node[color=blue,main node] (2) [above right of=1] {\textcolor{white}{\textbf{2}}};
  \node[color=blue,main node] (5) [below right of=1] {\textcolor{white}{\textbf{5}}};
  \node[color=blue,main node] (3) [right of=2] {\textcolor{white}{\textbf{3}}};
  \node[color=blue,main node] (4) [right of=5] {\textcolor{white}{\textbf{4}}};
  \node[color=blue,main node] (7) [below of=4] {\textcolor{white}{\textbf{7}}};
  \node[color=blue,main node] (6) [below of=5] {\textcolor{white}{\textbf{6}}};

  \path
    (1) edge [color=blue,loop left] node {} (1)
        edge[color=blue,bend left] node[above] {} (2)
    (2) edge[color=blue,bend left] node[above] {} (3)
	(3) edge[color=blue] node[above] {} (4)
		edge[color=blue,bend left] node[above] {} (2)
	(4) edge[color=blue] node[above] {} (5)
	(5) edge[color=blue,bend left] node[above] {} (1)
		edge[color=blue] node[above] {} (6)
	(6) edge[color=blue] node[below] {} (7)
	(7) edge[color=blue] node[below] {} (4);
\end{tikzpicture}
\item \begin{tikzpicture}[->, >=stealth',shorten >=1pt,auto,node distance=1.2cm,thick,main node/.style={circle,fill, fill opacity=0.75,minimum size = 6pt, inner sep = 1pt}]
  \node[color=blue,main node] (1) {\textcolor{white}{\textbf{6}}};
  \node[color=blue,main node] (2) [above right of=1] {\textcolor{white}{\textbf{7}}};
  \node[color=blue,main node] (7) [below right of=1] {\textcolor{white}{\textbf{5}}};
  \node[color=blue,main node] (3) [right of=2] {\textcolor{white}{\textbf{3}}};
  \node[color=blue,main node] (6) [right of=7] {\textcolor{white}{\textbf{4}}};
  \node[color=blue,main node] (4) [right of=3] {\textcolor{white}{\textbf{1}}};
  \node[color=blue,main node] (5) [right of=6] {\textcolor{white}{\textbf{2}}};

  \path
    (1) edge[color=blue,bend left] node[above] {} (2)
    (2) edge[color=blue] node[above] {} (3)
	(3) edge[color=blue] node[above] {} (6)
	(6) edge[color=blue] node[below] {} (7)
	    edge[color=blue] node[below] {} (5)
	(7) edge[color=blue,bend left] node[below] {} (1)
	(4) edge [color=blue,loop left] node {} (4)
		edge[color=blue,bend left] node[below] {} (5)
	(5) edge[color=blue,bend left] node[below] {} (4)
		edge[color=blue] node[below] {} (3);
\end{tikzpicture}
\item \begin{tikzpicture}[->, >=stealth',shorten >=1pt,auto,node distance=1.2cm,thick,main node/.style={circle,fill, fill opacity=0.75,minimum size = 6pt, inner sep = 1pt}]
  \node[color=blue,main node] (1) {\textcolor{white}{\textbf{7}}};
  \node[color=blue,main node] (2) [below of =1] {\textcolor{white}{\textbf{2}}};
  \node[color=blue,main node] (3) [right of=2] {\textcolor{white}{\textbf{4}}};
  \node[color=blue,main node] (4) [right of=3] {\textcolor{white}{\textbf{5}}};
  \node[color=blue,main node] (5) [below of=4] {\textcolor{white}{\textbf{6}}};
  \node[color=blue,main node] (6) [left of=5] {\textcolor{white}{\textbf{3}}};
  \node[color=blue,main node] (7) [left of=6] {\textcolor{white}{\textbf{1}}};

  \path
    (1) edge[color= blue,bend left] node[above] {} (2)
    (2) edge[color= blue,bend left] node[above] {} (1)
    	edge[color= blue] node[above] {} (6)
    (3)	edge[color= blue] node[above] {} (4)
    (4)	edge[color= blue] node[above] {} (5)
    (5)	edge[color= blue] node[above] {} (6)	
    (6)	edge[color= blue] node[above] {} (7)
    	edge[color= blue] node[above] {} (3)
    (7)	edge[color= blue] node[above] {} (2)
    	edge[color= blue,loop left] node[above] {} (7);
\end{tikzpicture}
\item
\begin{tikzpicture}[->, >=stealth',shorten >=1pt,auto,node distance=1.2cm,thick,main node/.style={circle,fill, fill opacity=0.75,minimum size = 6pt, inner sep = 1pt}]
  \node[color=blue,main node] (1) {\textcolor{white}{\textbf{5}}};
  \node[color=blue,main node] (2) [below of =1] {\textcolor{white}{\textbf{4}}};
  \node[color=blue,main node] (3) [right of=2] {\textcolor{white}{\textbf{3}}};
  \node[color=blue,main node] (4) [right of=3] {\textcolor{white}{\textbf{6}}};
  \node[color=blue,main node] (5) [below of=4] {\textcolor{white}{\textbf{7}}};
  \node[color=blue,main node] (6) [left of=5] {\textcolor{white}{\textbf{2}}};
  \node[color=blue,main node] (7) [left of=6] {\textcolor{white}{\textbf{1}}};

  \path
    (1) edge[color=blue,bend left] node[above] {} (2)
    (2) edge[color=blue,bend left] node[above] {} (1)
    	edge[color=blue] node[above] {} (7)
    (3)	edge[color=blue] node[above] {} (4)
    	edge[color=blue] node[above] {} (2)
    (4)	edge[color=blue] node[above] {} (5)
    (5)	edge[color=blue] node[above] {} (6)	
    (6)	edge[color=blue] node[above] {} (3)
    (7)	edge[color=blue] node[above] {} (6)
    	edge[color= blue,loop left] node[above] {} (7);
\end{tikzpicture}
\item
\begin{tikzpicture}[->, >=stealth',shorten >=1pt,auto,node distance=1.2cm,thick,main node/.style={circle,fill, fill opacity=0.75,minimum size = 6pt, inner sep = 1pt}]
  \node[color=blue,main node] (1) {\textcolor{white}{\textbf{4}}};
  \node[color=blue,main node] (2) [above right of=1] {\textcolor{white}{\textbf{1}}};
  \node[color=blue,main node] (7) [below right of=1] {\textcolor{white}{\textbf{3}}};
  \node[color=blue,main node] (3) [below right of=2] {\textcolor{white}{\textbf{2}}};
  \node[color=blue,main node] (4) [above right of=3] {\textcolor{white}{\textbf{5}}};
  \node[color=blue,main node] (6) [below right of=3] {\textcolor{white}{\textbf{7}}};
  \node[color=blue,main node] (5) [below right of=4] {\textcolor{white}{\textbf{6}}};

  \path
    (1) edge[color=blue] node[above] {} (2)
    	edge[color=blue,bend left] node[above] {} (7)
    (2) edge[color=blue] node[above] {} (3)
    	edge[color=blue,loop above] node[above] {} (2)
	(3) edge[color=blue] node[above] {} (7)
	(7) edge[color=blue,bend left] node[right] {} (1)
	(3) edge[color=blue] node[below] {} (4)
	(4) edge[color=blue] node[below] {} (5)
	(5) edge[color=blue] node[below] {} (6)
	(6) edge[color=blue] node[below] {} (3);
\end{tikzpicture}
\end{enumerate}
\end{multicols}
\caption{\bf 
List of potential digraphs with $7$ vertices and $10$ edges. \label{fig5}}
\end{figure}

%\newpage

\section{Calculations}
We now convert the graphs from Figure \ref{fig5} into properly signed matrices, and show that none of them can be realized by a stable matrix. First however, we prove the following lemma:
\begin{lemma}\label{ineq}
Let $A$ be a $7\times7$ real-valued matrix with the characteristic polynomial
\begin{equation}\label{PA}
  P_A(t)=t^7+c_1 t^6+c_2 t^5+c_3 t^4+c_4 t^3+c_5 t^2+c_6 t+c_7.
\end{equation}
If $A$ is stable, then all of the following inequalities must hold:
\begin{enumerate}
\item $c_2c_4-c_6>0$;
\item $c_1c_2-c_3>0$;
\item $c_1c_6-c_7>0$;
\item $c_2c_5-c_7>0$.
\end{enumerate}
\end{lemma}
\begin{proof}
By Lemma \ref{deg2}, a matrix $A$ is stable if and only if there exist $a_1,a_2,a_3,a_4,b_1,b_2,b_3>0$ such that
\begin{equation}\label{PA1}
P_A(t) = (t^2+b_1t+a_1)(t^2+b_2t+a_2)(t^2+b_3t+a_3)(t+a_4).
\end{equation}
Comparing the coefficients of \eqref{PA} and \eqref{PA1}, we have:
\begin{equation}
  \begin{split}
    c_1 =& a_4 + b_1 + b_2 + b_3, \\
    c_2 =& a_1 + a_2 + a_3 + a_4(b_1 + b_2+b_3) + b_1 b_2 + b_1 b_3 + b_2 b_3, \\
    c_3 =& a_1 a_4 + a_2 a_4 + a_3 a_4 + a_2 b_1 + a_3 b_1 + a_1 b_2 + a_3 b_2 + a_4 b_1 b_2\\
    &+a_1 b_3 + a_2 b_3 + a_4 b_1 b_3 + a_4 b_2 b_3 + b_1 b_2 b_3,\\
    c_4 =& a_1 a_2 + a_1 a_3 + a_2 a_3 + a_2 a_4 b_1 + a_3 a_4 b_1 + a_1 a_4 b_2 + a_3 a_4 b_2\\
    &+a_3 b_1 b_2 + a_1 a_4 b_3 + a_2 a_4 b_3 + a_2 b_1 b_3 + a_1 b_2 b_3 + a_4 b_1 b_2 b_3,\\
    c_5 =& a_1 a_2 a_4 + a_1 a_3 a_4 + a_2 a_3 a_4 + a_2 a_3 b_1 + a_1 a_3 b_2 + a_3 a_4 b_1 b_2\\
    &+a_1 a_2 b_3 + a_2 a_4 b_1 b_3 + a_1 a_4 b_2 b_3,\\
    c_6 =& a_1a_2a_3 + a_2a_3a_4b_1 + a_1 a_3 a_4 b_2 + a_1 a_2 a_4 b_3,\\
    c_7 =& a_1a_2a_3a_4.
  \end{split}
\end{equation}
We can verify that for each case of $c_i c_j-c_{i+j}$ listed in here,  $c_i c_j-c_{i+j}$ can be expressed as a sum of products of $a_i$ and $b_j$'s, hence $c_i c_j-c_{i+j}>0$ as all $a_i$ and $b_j$ are positive.
%We now proceed with each case:
%\begin{enumerate}
%\item Consider the product $c_2c_4$ and the term $c_6=a_1a_2a_3 + a_2a_3a_4b_1 + a_1 a_3 a_4 b_2 + a_1 a_2 a_4 b_3$.\\ In order to show that $c_2c_4-c_6$ must be positive, it is sufficient to note that:\\   $c_2c_4=(a_1a_2a_3 + a_2a_3a_4b_1 + a_1 a_3 a_4 b_2 + a_1 a_2 a_4 b_3)+k$ for some $k>0$.\\ Therefore $c_2c_4-c_6=c_6+k-c_6=k>0$.
%\item Consider the product $c_1c_2$ and the term\\$c_3=a_1 a_4 + a_2 a_4 + a_3 a_4 + a_2 b_1 + a_3 b_1 + a_1 b_2 + a_3 b_2 + a_4 b_1 b_2 +a_1 b_3 + a_2 b_3 + a_4 b_1 b_3 + a_4 b_2 b_3 + b_1 b_2 b_3$\\
% In order to show that $c_1c_2-c_3$ must be positive, it is sufficient to note that:\\    $c_1c_2=(a_1 a_4 + a_2 a_4 + a_3 a_4 + a_2 b_1 + a_3 b_1 + a_1 b_2 + a_3 b_2 + a_4 b_1 b_2 +a_1 b_3 + a_2 b_3 + a_4 b_1 b_3 + a_4 b_2 b_3 + b_1 b_2 b_3)+k$ for some $k>0$.\\ Therefore $c_1c_2-c_3=c_3+k-c_3=k>0$.
%\item Consider the product $c_1c_6$ and the term $c_7=a_1a_2a_3a_4$.\\ In order to show that $c_1c_6-c_7$ must be positive, it is sufficient to note that:\\    $c_1c_6=(a_1a_2a_3a_4)+k$ for some $k>0$.\\Therefore $c_1c_6-c_7=c_7+k-c_7=k>0$.
%\item Consider the product $c_2c_5$ and the term $c_7=a_1a_2a_3a_4$.\\ In order to show that $c_2c_5-c_7$ must be positive, it is sufficient to note that:\\    $c_2c_5=(a_1a_2a_3a_4)+k$ for some $k>0$.\\
%Therefore $c_2c_5-c_7=c_7+k-c_7=k>0$.
%\end{enumerate}
\end{proof}

Now we use Lemma \ref{ineq} to exclude all the $15$ digraphs (or equivalently sign patterns) in Figure \ref{fig5} to be potentially stable.

%These conditions are sufficient to eliminate each of the potential sign patterns below.\\
%With correct signs added, the matrices corresponding to the list of graphs generated above is as follows:\\
\begin{enumerate}
\item 
{\small $\begin{bmatrix}
-a_{11} & a_{12} & 0 & 0 & 0 & 0 & 0\\
0 & 0 & a_{23} & 0 & 0 & 0 & 0\\
0 & -a_{32} & 0 & a_{34} & 0 & 0 & 0\\
0 & 0 & 0 & 0 & a_{45} & 0 & 0\\
0 & 0 & 0 & 0 & 0 & a_{56} & 0\\
0 & 0 & 0 & -a_{64} & 0 & 0 & a_{67}\\
a_{71} & 0 & 0 & 0 & 0 & 0 & 0
\end{bmatrix}$}.

\medskip\noindent
Here we have:
$c_1=a_{11}$,
$c_2=a_{23}a_{32}$,
and $c_3=a_{11}a_{23}a_{32}+a_{45}a_{56}a_{64}$.
So $c_1c_2-c_3=-a_{45}a_{56}a_{64}<0$.
Thus this sign pattern is not potentially stable by Lemma \ref{ineq} part 2.
\item 
{\small $\begin{bmatrix}
-a_{11} & a_{12} & 0 & 0 & 0 & 0 & 0\\
0 & 0 & a_{23} & 0 & 0 & 0 & 0\\
0 & -a_{32} & 0 & a_{34} & 0 & 0 & 0\\
0 & 0 & 0 & 0 & a_{45} & 0 & 0\\
0 & 0 & 0 & 0 & 0 & a_{56} & 0\\
0 & 0 & 0 & 0 & 0 & 0 & a_{67}\\
-a_{71} & 0 & 0 & 0 & -a_{75} & 0 & 0
\end{bmatrix}$}.

\medskip\noindent
Here we have:
$c_1=a_{11}$,
$c_2=a_{23}a_{32}$,
$c_3=a_{11}a_{23}a_{32}+a_{56}a_{67}a_{75}$.
So $c_1c_2-c_3=-a_{56}a_{67}a_{75}<0$.
Thus this sign pattern is not potentially stable by Lemma \ref{ineq} part 2.
\item {\small $\begin{bmatrix}
-a_{11} & a_{12} & 0 & 0 & 0 & 0 & 0\\
0 & 0 & a_{23} & 0 & 0 & 0 & 0\\
0 & -a_{32} & 0 & a_{34} & 0 & 0 & 0\\
0 & 0 & 0 & 0 & a_{45} & 0 & 0\\
0 & 0 & 0 & 0 & 0 & a_{56} & 0\\
0 & 0 & 0 & 0 & 0 & 0 & a_{67}\\
\pm a_{71} & 0 & 0 & -a_{74} & 0 & 0 & 0
\end{bmatrix}$}.

\medskip\noindent
Here we have:\\
$c_2=a_{23}a_{32}$,
$c_4=a_{45}a_{56}a_{67}a_{74}$,
$c_6=a_{23}a_{32}a_{45}a_{56}a_{67}a_{74}$.
So $c_2c_4-c_6=0$.
Note that while both positive and negative values of $a_{71}$ allow for correct minors, the value of $a_{71}$ does not appear in our contradiction, and thus the contradiction holds regardless of the value of $a_{71}$. Thus this sign pattern is not potentially stable by Lemma \ref{ineq} part 1.
\item {\small $\begin{bmatrix}
-a_{11} & a_{12} & 0 & 0 & 0 & 0 & 0\\
0 & 0 & a_{23} & 0 & 0 & 0 & 0\\
0 & 0 & 0 & a_{34} & 0 & 0 & 0\\
0 & 0 & -a_{43} & 0 & a_{45} & 0 & 0\\
0 & 0 & 0 & 0 & 0 & a_{56} & 0\\
0 & 0 & 0 & 0 & 0 & 0 & a_{67}\\
-a_{71} & 0 & 0 & 0 & -a_{75} & 0 & 0
\end{bmatrix}$}.

\medskip\noindent
Here we have:
$c_1=a_{11}$,
$c_2=a_{34}a_{43}$,
$c_3=a_{11}a_{34}a_{43}+a_{56}a_{67}a_{75}$.
So $c_1c_2-c_3=-a_{56}a_{67}a_{75}<0$.
Thus this sign pattern is not potentially stable by Lemma \ref{ineq} part 2.
\item {\small $\begin{bmatrix}
-a_{11} & a_{12} & 0 & 0 & 0 & 0 & 0\\
0 & 0 & a_{23} & 0 & 0 & 0 & 0\\
0 & 0 & 0 & a_{34} & 0 & 0 & a_{37}\\
0 & 0 & 0 & 0 & a_{45} & 0 & 0\\
0 & 0 & 0 & 0 & 0 & a_{56} & 0\\
0 & 0 & 0 & 0 & 0 & 0 & a_{67}\\
-a_{71} & 0 & -a_{73} & 0 & 0 & 0 & 0
\end{bmatrix}$}.

\medskip\noindent
Here we have:
$c_1=a_{11}$,
$c_2=a_{37}a_{73}$,
$c_3=a_{11}a_{37}a_{73}$.
So $c_1c_2-c_3=0$.
Thus this sign pattern is not potentially stable by Lemma \ref{ineq} part 2.
\item {\small $\begin{bmatrix}
-a_{11} & a_{12} & 0 & 0 & 0 & 0 & 0\\
-a_{21} & 0 & a_{23} & 0 & 0 & 0 & 0\\
0 & 0 & 0 & a_{34} & 0 & 0 & 0\\
0 & 0 & 0 & 0 & a_{45} & 0 & 0\\
0 & 0 & -a_{53} & 0 & 0 & a_{56} & 0\\
0 & 0 & 0 & 0 & 0 & 0 & a_{67}\\
0 & -a_{72} & 0 & 0 & 0 & 0 & 0
\end{bmatrix}$}.

\medskip\noindent
Here we have:
$c_1=a_{11}$,
$c_6=a_{23}a_{43}a_{45}a_{56}a_{67}a_{72}$,
$c_7=a_{11}a_{23}a_{43}a_{45}a_{56}a_{67}a_{72}$.
So $c_1c_6-c_7=0$.
Thus this sign pattern is not potentially stable by Lemma \ref{ineq} part 3.
\item {\small $\begin{bmatrix}
-a_{11} & a_{12} & 0 & 0 & 0 & 0 & 0\\
-a_{21} & 0 & a_{23} & 0 & 0 & 0 & 0\\
0 & 0 & 0 & a_{34} & 0 & 0 & 0\\
0 & 0 & 0 & 0 & a_{45} & 0 & 0\\
0 & 0 & 0 & 0 & 0 & a_{56} & 0\\
0 & 0 & 0 & -a_{64} & 0 & 0 & a_{67}\\
0 & -a_{72} & 0 & 0 & 0 & 0 & 0
\end{bmatrix}$}.

\medskip\noindent
Here we have:
$c_1=a_{11}$,
$c_6=a_{23}a_{43}a_{45}a_{56}a_{67}a_{72}$,
$c_7=a_{11}a_{23}a_{43}a_{45}a_{56}a_{67}a_{72}$.
So $c_1c_6-c_7=0$.
Thus this sign pattern is not potentially stable by Lemma \ref{ineq} part 3.
%\newpage
\item {\small $\begin{bmatrix}
-a_{11} & a_{12} & 0 & 0 & 0 & 0 & 0\\
0 & 0 & a_{23} & 0 & 0 & 0 & 0\\
0 & 0 & 0 & a_{34} & 0 & 0 & 0\\
0 & 0 & 0 & 0 & a_{45} & 0 & 0\\
0 & -a_{52} & 0 & 0 & 0 & a_{56} & 0\\
-a_{61} & 0 & 0 & 0 & 0 & 0 & a_{67}\\
0 & 0 & 0 & 0 & 0 & -a_{76} & 0
\end{bmatrix}$}.

\medskip\noindent
Here we have:
$c_2=a_{67}a_{76}$,
$c_4=a_{23}a_{34}a_{45}a_{52}$,
$c_6=a_{67}a_{76}a_{23}a_{34}a_{45}a_{52}$.
So $c_2c_4-c_6=0$.
Thus this sign pattern is not potentially stable by Lemma \ref{ineq} part 2.
\item {\small $\begin{bmatrix}
-a_{11} & a_{12} & 0 & 0 & 0 & 0 & 0\\
0 & 0 & a_{23} & 0 & 0 & 0 & 0\\
0 & -a_{32} & 0 & a_{34} & 0 & 0 & 0\\
0 & 0 & 0 & 0 & a_{45} & 0 & 0\\
0 & 0 & 0 & 0 & 0 & a_{56} & 0\\
\pm a_{61} & 0 & 0 & 0 & 0 & 0 & a_{67}\\
0 & 0 & 0 & -a_{74} & 0 & 0 & 0
\end{bmatrix}$}.

\medskip\noindent
Here we have:
$c_2=a_{23}a_{32}$,
$c_5=a_{11}a_{45}a_{56}a_{67}a_{74}$,
$c_7=a_{23}a_{32}a_{11}a_{45}a_{56}a_{67}a_{74}$.
So $c_2c_5-c_7=0$.
Note that while both positive and negative values of $a_{61}$ allow for correct minors, the value of $a_{61}$ does not appear in our contradiction, and thus the contradiction holds regardless of the value of $a_{61}$. Thus this sign pattern is not potentially stable by Lemma \ref{ineq} part 4.
\item {\small $\begin{bmatrix}
-a_{11} & a_{12} & 0 & 0 & 0 & 0 & 0\\
0 & 0 & a_{23} & 0 & 0 & 0 & 0\\
0 & 0 & 0 & a_{34} & 0 & 0 & 0\\
0 & 0 & 0 & 0 & a_{45} & 0 & a_{47}\\
\pm a_{51} & 0 & 0 & 0 & 0 & a_{56} & 0\\
0 & 0 & 0 & 0 & -a_{65} & 0 & 0\\
0 & -a_{72} & 0 & 0 & 0 & 0 & 0
\end{bmatrix}$}.

\medskip\noindent
Here we have:
$c_2=a_{56}a_{65}$,
$c_4=a_{23}a_{34}a_{47}a_{72}$,
$c_6=a_{56}a_{65}a_{23}a_{34}a_{47}a_{72}$.
So $c_2c_4-c_6=0$.
Note that while both positive and negative values of $a_{51}$ allow for correct minors, the value of $a_{51}$ does not appear in our contradiction, and thus the contradiction holds regardless of the value of $a_{51}$. Thus this sign pattern is not potentially stable by Lemma \ref{ineq} part 2.
\item {\small $\begin{bmatrix}
-a_{11} & a_{12} & 0 & 0 & 0 & 0 & 0\\
0 & 0 & a_{23} & 0 & 0 & 0 & 0\\
0 & -a_{32} & 0 & a_{34} & 0 & 0 & 0\\
0 & 0 & 0 & 0 & a_{45} & 0 & 0\\
\pm a_{51} & 0 & 0 & 0 & 0 & a_{56} & 0\\
0 & 0 & 0 & 0 & 0 & 0 & a_{67}\\
0 & 0 & 0 & -a_{74} & 0 & 0 & 0
\end{bmatrix}$}.

\medskip\noindent
Here we have:
$c_2=a_{23}a_{32}$,
$c_4=a_{45}a_{56}a_{67}a_{74}$,
$c_6=a_{23}a_{32}a_{45}a_{56}a_{67}a_{74}$.
So $c_2c_4-c_6=0$.
Note that while both positive and negative values of $a_{51}$ allow for correct minors, the value of $a_{51}$ does not appear in our contradiction, and thus the contradiction holds regardless of the value of $a_{51}$. Thus this sign pattern is not potentially stable by Lemma \ref{ineq} part 2.
\item {\small $\begin{bmatrix}
-a_{11} & a_{12} & 0 & 0 & 0 & 0 & 0\\
-a_{21} & 0 & a_{23} & 0 & 0 & 0 & 0\\
0 & 0 & 0 & a_{34} & 0 & 0 & 0\\
0 & -a_{42} & 0 & 0 & a_{45} & 0 & 0\\
0 & 0 & 0 & 0 & 0 & a_{56} & 0\\
0 & 0 & 0 & 0 & 0 & 0 & a_{67}\\
0 & 0 & -a_{73} & 0 & 0 & 0 & 0
\end{bmatrix}$}.

\medskip\noindent
Here we have:
$c_2=a_{12}a_{21}$,
$c_5=a_{34}a_{45}a_{56}a_{67}a_{73}$,
$c_7=a_{12}a_{21}a_{34}a_{45}a_{56}a_{67}a_{73}$.
So $c_2c_5-c_7=0$.
Thus this sign pattern is not potentially stable by Lemma \ref{ineq} part 4.
\item {\small $\begin{bmatrix}
-a_{11} & a_{12} & 0 & 0 & 0 & 0 & 0\\
0 & 0 & a_{23} & 0 & 0 & 0 & -a_{27}\\
\pm a_{31} & 0 & 0 & a_{34} & 0 & 0 & 0\\
0 & 0 & 0 & 0 & a_{45} & 0 & 0\\
0 & 0 & 0 & 0 & 0 & a_{56} & 0\\
0 & 0 & -a_{63} & 0 & 0 & 0 & 0\\
0 & -a_{72} & 0 & 0 & 0 & 0 & 0
\end{bmatrix}$}.

\medskip\noindent
Here we have:
$c_2=a_{27}a_{72}$,
$c_4=a_{34}a_{45}a_{56}a_{63}$,
$c_6=a_{27}a_{72}a_{34}a_{45}a_{56}a_{63}$.
So $c_2c_4-c_6=0$.
Note that while both positive and negative values of $a_{31}$ allow for correct minors, the value of $a_{31}$ does not appear in our contradiction, and thus the contradiction holds regardless of the value of $a_{31}$. Thus this sign pattern is not potentially stable by Lemma \ref{ineq} part 2.
\item {\small $\begin{bmatrix}
-a_{11} & a_{12} & 0 & 0 & 0 & 0 & 0\\
0 & 0 & a_{23} & 0 & 0 & 0 & 0\\
0 & 0 & 0 & a_{34} & 0 & a_{36} & 0\\
\pm a_{41} & 0 & 0 & 0 & a_{45} & 0 & 0\\
0 & 0 & 0 & -a_{54} & 0 & 0 & 0\\
0 & 0 & 0 & 0 & 0 & 0 & a_{67}\\
0 & -a_{72} & 0 & 0 & 0 & 0 & 0
\end{bmatrix}$}.

\medskip\noindent
Here we have:
$c_2=a_{45}a_{54}$,
$c_5=a_{11}a_{23}a_{36}a_{67}a_{72}$,
$c_7=a_{45}a_{54}a_{11}a_{23}a_{36}a_{67}a_{72}$.
So $c_2c_5-c_7=0$.
Note that while both positive and negative values of $a_{41}$ allow for correct minors, the value of $a_{41}$ does not appear in our contradiction, and thus the contradiction holds regardless of the value of $a_{41}$. Thus this sign pattern is not potentially stable by Lemma \ref{ineq} part 4.
\item {\small $\begin{bmatrix}
-a_{11} & a_{12} & 0 & 0 & 0 & 0 & 0\\
0 & 0 & a_{23} & 0 & a_{25} & 0 & 0\\
0 & 0 & 0 & a_{34} & 0 & 0 & 0\\
\pm a_{41} & 0 & -a_{43} & 0 & 0 & 0 & 0\\
0 & 0 & 0 & 0 & 0 & a_{56} & 0\\
0 & 0 & 0 & 0 & 0 & 0 & a_{67}\\
0 & -a_{72} & 0 & 0 & 0 & 0 & 0
\end{bmatrix}$}.

\medskip\noindent
Here we have:
$c_2=a_{43}a_{34}$,
$c_5=a_{11}a_{25}a_{56}a_{67}a_{72}$,
$c_7=a_{43}a_{34}a_{11}a_{25}a_{56}a_{67}a_{72}$.
So $c_2c_5-c_7=0$.
Note that while both positive and negative values of $a_{41}$ allow for correct minors, the value of $a_{41}$ does not appear in our contradiction, and thus the contradiction holds regardless of the value of $a_{41}$. Thus this sign pattern is not potentially stable by Lemma \ref{ineq} part 4.
\end{enumerate}

\end{document}